\documentclass[number, sort&compress]{elsarticle}

\journal{Stochastic Processes and their Applications}

\usepackage{mathptmx}
\usepackage{amsmath}
\usepackage{amsthm}
\usepackage[colorlinks]{hyperref}

\usepackage{amssymb}

\DeclareMathOperator {\E}{E}
\DeclareMathOperator{\Var}{Var}

\DeclareMathOperator{\arcsech}{arcsech}

\newcommand{\secx}{\mathsf X}
\newcommand{\secu}{\mathsf U}
\newcommand{\secz}{\mathsf Z}

\newtheorem{proposition}{Proposition}
\newtheorem{theorem}{Theorem}
\newtheorem{corollary}{Corollary}
\newtheorem{algorithm}{Algorithm}
\newtheorem{lemma}{Lemma}

\begin{document}

\begin{frontmatter}

\title{An asymptotically optimal Bernoulli factory for certain functions that can be expressed as power series}

\author{Luis Mendo}

\address{
Information Processing and Telecommunications Center,
Universidad Polit\'ecnica de Madrid.\\
Avenida Complutense, 30. 28040 Madrid, Spain.\\
E-mail: \href{mailto:luis.mendo@upm.es}{\nolinkurl{luis.mendo@upm.es}}
}


\begin{abstract}
Given a sequence of independent Bernoulli variables with unknown parameter $p$, and a function $f$ expressed as a power series with non-negative coefficients that sum to at most $1$, an algorithm is presented that produces a Bernoulli variable with parameter $f(p)$. In particular, the algorithm can simulate $f(p)=p^a$, $a\in(0,1)$. For functions with a derivative growing at least as $f(p)/p$ for $p\rightarrow 0$, the average number of inputs required by the algorithm is asymptotically optimal among all simulations that are fast in the sense of Nacu and Peres. A non-randomized version of the algorithm is also given. Some extensions are discussed.
\end{abstract}


\end{frontmatter}

\section{Introduction}
\label{part: intro}

Let $\secx = (X_i)_{i \in \mathbb N}$ denote a sequence of independent, identically distributed (i.i.d.) Bernoulli random variables $X_i$ with unknown parameter $p$. A \emph{non-randomized stopping rule} on $\secx$ is a sequence of \emph{stopping functions} $\tau_1(x_1)$, $\tau_2(x_1, x_2)$, \ldots, $\tau_i (x_1, \ldots, x_i)$, \ldots with $\tau_i: \{0,1\}^i \to \{0,1\}$. Given a realization $x_1,x_2,\ldots$ of $\secx$, the \emph{stopping time} $N$ is the smallest integer $i$ for which $\tau_i(x_1, \ldots, x_i)$ equals $1$, or infinity if this does not occur. $N$ is assumed to be finite almost surely; that is, for almost all realizations of $\secx$ at least one of the stopping functions takes the value $1$. 

A \emph{randomized stopping rule} uses, in addition to $\secx$, a \emph{randomizing sequence} $\secu = (U_i)_{i \in \mathbb N}$ of independent random variables with uniform distribution on $(0,1$). The stopping function $\tau_i$ depends on $x_1,u_1; x_2,u_2;\ldots;x_i,u_i$, where each $u_j$ represents the value taken by the random variable $U_j$; and is assumed to be measurable. The stopping time $N$ is defined as before, and is again assumed to be finite almost surely.

Another possible definition of a randomizing stopping rule would be to specify that at each $i$, given $X_1=x_1,\ldots,X_i=x_i$, there is a certain probability of stopping that depends on $x_1,\ldots,x_i$. This corresponds to the above definition if the output of $\tau_i$ is obtained from comparing $U_i$ with a threshold that depends on $x_1,\ldots,x_i$. Thus the definition based on the auxiliary sequence $\secu$ captures all the randomness that can be effected by a randomized stopping rule (and is more convenient for the purposes of this paper).

Let $f: S \to (0,1)$ denote a function defined on a set $S \subseteq (0,1)$, and let $\secx$ be a sequence of independent Bernoulli variables with parameter $p \in S$. A \emph{non-randomized Bernoulli factory} of $f$ based on $\secx$ is an algorithm that, using values from the sequence $\secx$ as inputs, generates a Bernoulli variable $Y$ with parameter $f(p)$. Specifically, the number $N$ of required inputs is a stopping time on $\secx$ dictated by a non-randomized stopping rule; and the output value $Y$ depends on $X_1, \ldots, X_N$.

A \emph{randomized Bernoulli factory} uses, in addition to the sequence $\secx$, an auxiliary sequence $\secu$ of independent random variables uniformly distributed on $(0,1)$. Specifically, $N$ is given by a randomized stopping rule on $\secx$ with $\secu$ as randomizing sequence. The output $Y$ is also possibly randomized, that is, it depends on $X_1, \ldots, X_N$ and $U_1, \ldots, U_N$. Specifically, there exists a sequence of functions $\gamma_1, \gamma_2, \ldots$ such that for $N = n$ and for $X_1=x_1$, $U_1 = u_1$; \ldots; $X_n=x_n$, $U_n = u_n$ the output is given as $\gamma_n(x_1,u_1; \ldots; x_n,u_n)$. These functions are assumed to be measurable.

The use of the term ``randomized'' in the above definitions of randomized stopping rules or Bernoulli factories refers to the fact that the stopping time $N$ and the factory output $Y$ are random even if conditioned on the input sequence $\secx$, due to the additional source of randomness represented by $\secu$. In the following, a Bernoulli factory will also be referred to as a \emph{simulation}.

One of the earliest references of a Bernoulli factory, for the specific case $f(p)=1/2$, appears in a work by von Neumann \cite{VonNeumann51}.
For general $f$, Keane and O'Brien \cite{Keane94} give a necessary and sufficient condition for a simulation of $f$ to be possible, namely that the function either is constant, or is continuous and satisfies a certain polynomial bound.
Nacu and Peres \cite{Nacu05} define a non-randomized simulation to be \emph{fast} if the distribution of $N$ has an exponential tail, that is, for any $p \in S$ there exist values $A>0$, $\beta<1$ (which may depend on $p$) such that
\begin{equation}
\label{eq: exp tail}
\Pr[N>n] \leq A \beta^n.
\end{equation}
The authors prove that a fast simulation exists for any $f$ real analytic on any closed interval contained in $(0,1)$. In this paper, the definition  will be extended to randomized simulations, which will be considered fast if they satisfy \eqref{eq: exp tail}.

Several works on Bernoulli factories present simulation algorithms for specific functions. Considerable attention has been given to $f(p) = \min\{2p, 1-2\epsilon\}$, $\epsilon>0$, which is an important building block for simulating other functions \cite{Nacu05, Latuszynski11}; and to linear functions $f(p) = cp$, $c>1$ defined on a suitable set $S \subset (0,1/c)$; see the work by Huber \cite{Huber16, Huber17}.

A crucial parameter of a Bernoulli factory is $\E[N]$, that is, how many inputs $X_i$ are required on average to generate a sample of $Y$. In applications, observing the variables $X_i$ is usually costly, and thus $\E[N]$ should be as small as possible. For randomized Bernoulli factories, the auxiliary random variables $U_i$ are assumed to be cost-free.

This paper deals mainly with functions of the form $f(p) = 1-\sum_{k=1}^\infty c_k (1-p)^k$ where the coefficients $c_k$ are non-negative and sum to $1$.
A randomized algorithm to simulate any such function is presented in Section \ref{part: algo, E van}. The algorithm is shown to be fast in the sense of \eqref{eq: exp tail}, and its average number of inputs, $\E[N]$, is computed. The algorithm can be particularized to functions $f(p)=p^a$, $a \in(0,1)$. For $a=1/2$ the algorithm is similar to that given by W{\"a}stlund \cite{Wastlund99}; and the presented results affirmatively answer question 1 from \cite{Nacu05}, i.e.\ establish that $\sqrt p$ can be simulated with finite $\E[N]$.

An interesting subclass of functions is formed by those that, in addition to having a power series expression as above, satisfy the following two conditions: $f(p)/p \rightarrow \infty$ as $p \rightarrow 0$, and the derivative $f'(p)$ asymptotically grows at least as fast as $f(p)/p$. In particular, this includes all functions that behave asymptotically like $b p^a$, $a\in (0,1)$, $b \in (0,\infty)$. For these functions, it will be seen that any fast simulation algorithm has an average number of inputs $\E[N]$ that grows without bound as $p \rightarrow 0$. Therefore it is important to analyse the asymptotic rate of growth of $\E[N]$. This analysis is presented in Section \ref{part: as opt}. The results show that the proposed algorithm is asymptotically optimal for the mentioned subclass of functions, in the sense that for any other fast algorithm $\E[N]$ grows at least as fast with $p$.

A non-randomized version of the proposed algorithm is given in Section \ref{part: algo non-rand}, and is also shown to be fast. Section \ref{part: extensions} discusses some extensions of the algorithms to cover a broader range of functions. Section \ref{part: concl, future} presents conclusions and open problems. Section \ref{part: proofs} contains proofs to all results.

The following notation is used throughout the paper. $x^{(m)}$ represents the rising factorial $x(x+1)\cdots(x+m-1)$ for $m \in \mathbb N$, and $x^{(0)}=0$. For $m \in \mathbb N$, $m!!$ denotes the double factorial, that is, $m(m-2)(m-4)\cdots 2$ if $m$ is even, or $m(m-2)(m-4)\cdots 1$ if $m$ is odd. Given two positive functions $f_1$ and $f_2$, $f_1(x)$ is said to be $\Omega(f_2(x))$ for $x \rightarrow x_0$ if there are constants $C, \delta > 0$ such that $f_1(x) \geq C f_2(x)$ for all $x$ such that $|x-x_0|<\delta$.

\section{Simulation algorithm. Average number of inputs}
\label{part: algo, E van}

Consider an i.i.d.\ sequence $\secx$ of random variables $X_i$ that take the value $1$ with probability $p$ and $0$ with probability $1-p$. Let $f:(0,1)\to(0,1)$ be a function that can be expressed as a power series
\begin{equation}
\label{eq: f 1-p}
f(p) = 1-\sum_{k=1}^\infty c_k (1-p)^k
\end{equation}
with
\begin{align}
\label{eq: c k non-neg}
c_k & \geq 0, \\
\label{eq: sum c k}
\sum_{k=1}^\infty c_k & = 1.
\end{align}
Note that this implies that $f$ is differentiable, and $\lim_{p \rightarrow 0} f(p) = 0$, $\lim_{p \rightarrow 1} f(p) = 1$.

The randomized algorithm to be presented yields a Bernoulli random variable $Y$ with parameter $f(p)$. It takes as inputs a number of random variables from the sequence $\secx$, as well as from an auxiliary sequence $\secu$ of i.i.d.\ random variables $U_i$ uniformly distributed on $(0,1)$ and independent from the $X_i$ variables. The algorithm makes use of coefficients $d_k$ computed from $c_k$ as follows:
\begin{equation}
\label{eq: d k c k}
d_k = \frac{c_k}{1-\sum_{j=1}^{k-1} c_j}.
\end{equation}
From \eqref{eq: c k non-neg} and \eqref{eq: sum c k} it stems that $0 \leq d_k \leq 1$. If the number of non-zero coefficients $c_k$ is finite, i.e.\ if there exists $K$ such that $c_K > 0$ and $c_k = 0$ for $k>K$, \eqref{eq: d k c k} gives $d_K = 1$; and for $k>K$ the coefficient $d_k$ is not defined (and is not necessary, as will be seen).

\begin{algorithm}
\label{algo: Bernoulli factory rand}
Let $f$ be given by \eqref{eq: f 1-p}--\eqref{eq: sum c k},
and let $d_k$ be defined by \eqref{eq: d k c k}. The input of the algorithm is a sequence $\secx$ of i.i.d.\ Bernoulli random variables. The output is a Bernoulli random variable $Y$.
\begin{enumerate}
\item
\label{step: algo rand, start}
Set $i = 1$.
\item
\label{step: algo rand, loop}
\label{step: algo rand, input}
Take one input $X_i$.
\item
\label{step: algo rand, U n V n}
Produce $U_i$ uniform on $(0,1)$. Let $V_i=1$ if $U_i < d_i$ or $V_i=0$ otherwise.
\item
\label{step: algo rand, output}
\label{step: algo rand, end}
If $V_i$ or $X_i$ are $1$, output $Y = X_i$ and finish. Else increase $i$ and go back to step \ref{step: algo rand, loop}.
\end{enumerate}
\end{algorithm}

The idea of this algorithm is similar to that presented by W{\"a}stlund \cite{Wastlund99} to simulate $f(p) = \sqrt{p}$, namely, decompose the event $Y=0$ as an infinite sum of mutually exclusive events, each with probability $c_k (1-p)^k$. However, there are two differences here. First, the referenced paper treats the $c_k$ and $(1-p)^k$ parts separately. Namely, an auxiliary random variable $L$ is first generated with $\Pr[L=k] = c_k$, $k \in \mathbb N$. This variable represents the amount of inputs $X_i$ that need to be taken. Then, $(1-p)^L$ is simulated using the product $\prod_{i=1}^L (1-X_i)$. On the other hand, Algorithm \ref{algo: Bernoulli factory rand} reduces the number of required inputs $X_i$ by stopping as soon as one of the $X_i$ variables is $1$. This can be done because in that case the above product is $0$ irrespective of the values of the remaining $X_i$ variables.

The second difference is that Algorithm \ref{algo: Bernoulli factory rand} uses auxiliary Bernoulli variables $V_i$ with respective parameters $d_i$, instead of a random variable $L$ with the distribution given by coefficients $c_k$. The latter approach, used in \cite{Wastlund99}, was probably motivated by the fact that for $f(p) = \sqrt{p}$ the coefficients $c_k$ in \eqref{eq: f 1-p} are
\begin{equation}
\label{eq: c_k sqrt}
c_k = \frac{1}{2^{2k-1} k} \binom{2k-2}{k-1}
\end{equation}
and thus simulating $L$ is particularly easy, as it lends itself to a simple probabilistic interpretation. Specifically, if a fair coin is flipped until the total number of tails exceeds the total number of heads, the probability that this happens after $2k-1$ steps is precisely \eqref{eq: c_k sqrt}. For other functions it is still possible to simulate $L$ from fair coin flips, or from a uniform random variable, as long as conditions \eqref{eq: c k non-neg} and \eqref{eq: sum c k} are satisfied; but a simple probabilistic experiment analogous to that for $f(p) = \sqrt p$ may not exist. The Bernoulli random variables $V_i$ provide an alternative, which can also be used for any coefficients $c_k$ that satisfy the indicated conditions. Effectively, each $d_k$ represents the conditional probability that $L = k$ given that $L \geq k$. The following proposition clarifies this.
\begin{proposition}
\label{prop: c k d k}
The coefficients $d_k$ defined by \eqref{eq: d k c k} satisfy
\begin{equation}
\label{eq: c k d k}
c_k = d_k \prod_{j=1}^{k-1} (1-d_j).
\end{equation}
\end{proposition}
Additionally, this interpretation of $d_k$ as a conditional probability makes it clear that if $d_K = 1$ for some $K$ (which occurs if the series in \eqref{eq: f 1-p} is finite with $c_K > 0$, $c_k=0$ for $k > K$) it is unnecessary to define coefficients $d_k$ for $k>K$.

The algorithm can also be phrased as a particular case of the reverse-time martingale approach of \L{}atuszy{\'n}ski, Kosmidis, Papaspiliopoulos and Roberts with random bounds \cite[algorithm 3]{Latuszynski11}. Specifically, from \eqref{eq: f 1-p} it is possible to obtain monotone sequences of random upper bounds and lower bounds that depend on the observed inputs, such that the sufficient conditions for the referenced algorithm are satisfied. This approach has the additional advantage that condition \eqref{eq: sum c k} is not required. Note, however, that this restriction of Algorithm \ref{algo: Bernoulli factory rand} is immaterial, because the coefficients can always be scaled to sum $1$ and then the output $Y$ can be multiplied by an independent Bernoulli variable with parameter equal to the desired sum (see Section \ref{part: extensions}).

\begin{theorem}
\label{teo: Pr Y}
Given a sequence $\secx$ of i.i.d.\ Bernoulli random variables with parameter $p \in (0,1)$, a function $f$ of the form \eqref{eq: f 1-p}--\eqref{eq: sum c k}, 
and coefficients $d_k$ computed from \eqref{eq: d k c k}, the Bernoulli random variable $Y$ produced by Algorithm \ref{algo: Bernoulli factory rand} has $\Pr[Y=1] = f(p)$.
\end{theorem}

Algorithm \ref{algo: Bernoulli factory rand} takes a new input for each iteration $i$. Thus the number of used inputs, $N$, coincides with the value of $i$ when the algorithm finishes.

\begin{theorem}
\label{teo: E[N], exp tail}
For $f$ given by \eqref{eq: f 1-p}--\eqref{eq: sum c k} and $p \in (0,1)$, 
the average number of inputs required by Algorithm \ref{algo: Bernoulli factory rand} is
\begin{equation}
\E[N] = \frac{f(p)}{p}.
\end{equation}
In addition, the algorithm is fast in the sense of \eqref{eq: exp tail}.
\end{theorem}

It is interesting to consider the case $f(p) = p^a$, $a \in (0,1)$. This can be expressed in the form \eqref{eq: f 1-p} with
\begin{equation}
\label{eq: coef p a}
c_k = \frac{(1-a)^{(k-1)} a}{k!},
\end{equation}
from which $d_k = a/k$. Algorithm \ref{algo: Bernoulli factory rand} can be applied, and $\E[N] = p^{a-1}$. In particular, $f(p) = \sqrt{p}$ can be simulated with $\E[N]=1/\sqrt{p}$. This solves \cite[question 1]{Nacu05}, which asks if there is an algorithm that simulates $\sqrt{p}$ on $(0,1)$ for which the number of required inputs has finite expectation for all $p$.

\section{Asymptotic optimality}
\label{part: as opt}

A natural question is whether the average number of inputs required by Algorithm \ref{algo: Bernoulli factory rand} can be improved by some other algorithm. The following proposition and theorem are useful steps towards the answer.

\begin{proposition}
\label{prop: E[N] cont p}
Given an open set $S \subseteq (0,1)$, consider a function $f: S \to (0,1)$ and a (possibly randomized) Bernoulli factory for $f$ that is fast, as defined by \eqref{eq: exp tail}. Then $\E[N]$ is finite and is a continuous function of $p \in S$.
\end{proposition}

This proposition, which  will be used to prove Theorem \ref{teo: E[N] bound}, can be given an intuitive interpretation as follows. Since the stopping functions only depend on the values of the sequences $\secx$ and $\secu$ and the distribution of those values varies smoothly (or is constant) with $p$, it seems reasonable to expect $\E[N]$ to be a continuous function of $p$. As established by the proposition, fastness of the Bernoulli factory is indeed sufficient to ensure this.

\begin{theorem}
\label{teo: E[N] bound}
Consider an open set $S \subseteq (0,1)$, a differentiable function $f: S \to (0,1)$ and a (possibly randomized) Bernoulli factory for $f$ that is fast in the sense of \eqref{eq: exp tail}. Then
\begin{equation}
\label{eq: E[N] bound}
\E[N] \geq (f'(p))^2 \frac{p(1-p)}{f(p)(1-f(p))}.
\end{equation}
\end{theorem}

The main idea in the proof of this theorem is as follows. Given an arbitrary, possibly randomized factory for $f$ that uses the input sequence $\secx$ with parameter $p$, the output $Y$ can be seen as a sequential estimator of $f(p)$. Wolfowitz's extension of the Cram{\'e}r-Rao bound to sequential estimators \cite{Wolfowitz47} can be applied to $Y$. This sets a lower bound on $\Var[Y]$ that depends on $\E[N]$. Comparing with the actual variance gives the claimed lower bound on $\E[N]$.

The proof technique has some similarities with those used by Huber \cite{Huber16, Huber17}. Specifically, the method employed in \cite{Huber16} to establish a lower bound on $\E[N]$ for linear functions $f(p) = cp$, $c>1$, $p \in (0,(1-\epsilon)/c)$ is also based on considering the Bernoulli factory as an estimator of $f(p)$; but instead of the Cram{\'e}r-Rao inequality, a different bound is used which relates $\E[N]$ to the confidence level for a given interval. In \cite{Huber17} the Cram{\'e}r-Rao inequality is used, albeit informally, to provide evidence for a lower bound on $\E[N]$ for linear functions.

In order to compare Algorithm \ref{algo: Bernoulli factory rand} with others in terms of $\E[N]$, the most interesting case is that of functions $f$ for which this algorithm gives $\E[N] \rightarrow \infty$ as $p \rightarrow 0$; that is, functions of the form \eqref{eq: f 1-p}--\eqref{eq: sum c k} with
\begin{equation}
\label{eq: cond f(p)}
\lim_{p \rightarrow 0} \frac{f(p)}{p} = \infty.
\end{equation}
In this case, since the average number of inputs used by Algorithm \ref{algo: Bernoulli factory rand} grows without bound, it is important to know if the growth rate could be reduced by using some other algorithm. As will be established by Theorem \ref{teo: E[N] asympt opt}, for a certain subclass of these functions the average number of inputs required by any fast algorithm increases, as $p \rightarrow 0$, at least as fast as it does with Algorithm \ref{algo: Bernoulli factory rand}, which is thus asymptotically optimal.

Consider the class of functions $f: (0,1) \to (0,1)$ given by \eqref{eq: f 1-p}--\eqref{eq: sum c k} that satisfy \eqref{eq: cond f(p)} and 
\begin{equation}
\label{eq: cond f'(p)}
f'(p) = \Omega(f(p)/p) \quad \text{ for } p \rightarrow 0.
\end{equation}
Conditions \eqref{eq: cond f(p)} and \eqref{eq: cond f'(p)} mean that, asymptotically, $f(p)$ increases faster than $p$ and $f'(p)$ increases at least as fast as $f(p)/p$. In particular, they are fulfilled by any differentiable function $f$ such that
\begin{equation}
\label{eq: cond f a b}
\lim_{p \rightarrow 0} \frac{f(p)}{p^a} = b \quad \text{ for some } a \in (0,1),\ b \in (0,\infty).
\end{equation}
Indeed, it is clear that \eqref{eq: cond f(p)} holds if \eqref{eq: cond f a b} does. As for \eqref{eq: cond f'(p)}, observe that \eqref{eq: cond f a b} implies $\lim_{p \rightarrow 0} f(p) = 0$, and thus by L'H\^opital's rule
\begin{equation}
\label{eq: cond f a b 2}
\lim_{p \rightarrow 0} f'(p) \, p^{1-a} = ab.
\end{equation}
Dividing \eqref{eq: cond f a b 2} by \eqref{eq: cond f a b} it is seen that $\lim_{p \rightarrow 0} f'(p) p / f(p) = a$, which implies \eqref{eq: cond f'(p)}.

Some examples of functions of the form \eqref{eq: f 1-p}--\eqref{eq: sum c k} for which \eqref{eq: cond f(p)} and \eqref{eq: cond f'(p)} hold are given by the next proposition.
\begin{proposition}
\label{prop: ex}
The following functions can be expressed as in \eqref{eq: f 1-p}--\eqref{eq: sum c k} and satisfy \eqref{eq: cond f(p)} and \eqref{eq: cond f'(p)}:
\begin{align}
\label{eq: p a}
f(p) &= p^a \quad \text{for } a \in (0,1) \\
\label{eq: ex frac sqrt}
f(p) &= \frac{2 \sqrt p}{1+\sqrt p} \\
\label{eq: ex log sqrt}
f(p) &= \log_2(1+\sqrt p) \\
\label{eq: ex exp sqrt}
f(p) &= \frac{1-e^{-\sqrt{p}}}{1-e^{-1}}\\
\label{eq: ex p log p}
f(p) &= p(1-\log p).
\end{align}
\end{proposition}
Note that functions \eqref{eq: p a}--\eqref{eq: ex exp sqrt}
satisfy the more specialized condition \eqref{eq: cond f a b}, whereas \eqref{eq: ex p log p} does not.

The following theorem and its corollary establish that, for the class of functions defined above, Algorithm \ref{algo: Bernoulli factory rand}  is asymptotically optimum among all Bernoulli factories that are fast in the sense of Nacu and Peres.

\begin{theorem}
\label{teo: E[N] asympt opt}
Let $S \subseteq (0,1)$ be an open set that has $0$ as a limit point. Consider a differentiable function $f: S \to (0,1)$ for which \eqref{eq: cond f'(p)} holds. Any (possibly randomized) Bernoulli factory for $f$ that is fast in the sense of \eqref{eq: exp tail} satisfies
\begin{equation}
\label{eq: teo: E[N] asympt opt}
\E[N] = \Omega(f(p)/p) \quad \text{ for } p \rightarrow 0.
\end{equation}
\end{theorem}

Note that this theorem does not require \eqref{eq: cond f(p)}. However, for functions that do not satisfy this condition the result \eqref{eq: teo: E[N] asympt opt} is less interesting, and indeed a stronger bound on $\E[N]$ can be found. Namely, for a non-constant function any algorithm needs to take at least one input from $\secx$, and thus $\E[N] \geq 1$, which is $\Omega(f(p)/p)$ if \eqref{eq: cond f(p)} does not hold. On the other hand, a constant function only satisfies the hypotheses of the theorem if it is the null function; and in any case, constant functions can be simulated by  a randomized algorithm without observing $\secx$ (only $\secu$ is needed).

\begin{corollary}
\label{corol: E[N] asympt opt}
For any function $f$ that can be expressed as \eqref{eq: f 1-p}--\eqref{eq: sum c k}
and satisfies conditions \eqref{eq: cond f(p)} and \eqref{eq: cond f'(p)}, Algorithm \ref{algo: Bernoulli factory rand} is asymptotically optimal as $p \rightarrow 0$ among all fast algorithms; that is, for any algorithm that satisfies condition \eqref{eq: exp tail}, there exist $C, \delta >0$ such that $\E[N] \geq C f(p)/p$ for all $p < \delta$.
\end{corollary}

The results presented in this section are somewhat related to other results that have appeared in previous works. Elias \cite{Elias72} considers a non-randomized Bernoulli factory for the function $f(p) = 1/2$, obtains an entropy-based bound on the average number of outputs per input, and gives an algorithm that approaches that bound. Peres \cite{Peres92} shows that iterating von Neumann's procedure achieves the same efficiency. Stout and Warren \cite{Stout84} carry out a similar analysis for the average number of inputs per output required for simulating  $f(p) = 1/2$, and also propose several algorithms. Huber \cite{Huber16, Huber17}, as previously mentioned, considers linear functions $f(p) = cp$, $c>1$, $p \in (0,(1-\epsilon)/c)$, and gives upper and lower bounds on the average number of inputs per output.

For constant functions Theorem \ref{teo: E[N] bound} reduces to the trivial $\E[N] \geq 0$. On the other hand, for $f(p) = cp$ it yields $c(1-p)/(1-cp)$ as a lower bound for $\E[N]$. This is the bound that was conjectured in \cite[section 4]{Huber17} based on an informal argument, which has thus been formalized (and generalized) by Theorem \ref{teo: E[N] bound}. As for Theorem \ref{teo: E[N] asympt opt}, it cannot be directly compared with the bounds in the above referenced works because, as previously mentioned, the theorem does not apply to the constant function $f(p)=1/2$, and for linear functions it reduces to the trivial $\liminf_{p \rightarrow 0} \E[N] > 0$.

\section{Non-randomized algorithm}
\label{part: algo non-rand}

A non-randomized version of Algorithm \ref{algo: Bernoulli factory rand} will be given next. Instead of using an auxiliary variable $U_i$ to produce a Bernoulli variable $V_i$ with parameter $d_i$ in step \ref{step: algo rand, U n V n}, the required $V_i$ is obtained from additional input samples, using a well known procedure based on the binary expansion of $d_i$ \cite[proof of proposition 13]{Nacu05}. Note also that a variation of \cite[algorithm 3]{Latuszynski11} could be used for the same purpose, with progressively finer truncations of the binary representation of $d_i$ providing the lower and upper bounds required therein. Consider the binary expansion of $d_i \in [0,1]$ assuming zero as integer part and an infinite amount of digits in the fractional part. Thus the decimal values $0$, $0.75$ and $1$ are respectively expressed in binary as $0.000\cdots$, $0.11000\cdots$ (or equivalently $0.10111\cdots$) and $0.11111\cdots$.

\begin{algorithm}
\label{algo: Bernoulli factory non-rand}
The algorithm uses the same steps \ref{step: algo rand, start}--\ref{step: algo rand, end} from Algorithm \ref{algo: Bernoulli factory rand} except that step \ref{step: algo rand, U n V n} is replaced by the following:
{
\renewcommand{\theenumi}{3.\arabic{enumi}}
\begin{enumerate}
\item
\label{step: algo non-rand, start}
Set $j=1$.
\item
\label{step: algo non-rand, loop}
Keep taking pairs of values from the sequence $\secx$ until the two values of a pair are different. Let $T$ be the first value of that pair.
\item
\label{step: algo non-rand, branch}
If $T=0$ increase $j$ and go back to step \ref{step: algo non-rand, loop}. Else set $V_i$ equal to the $j$-th digit in the fractional part of the binary expansion of $d_i$.
\end{enumerate}
}
\end{algorithm}

The total number of inputs taken from $\secx$ is obviously greater than with Algorithm \ref{algo: Bernoulli factory rand}. However, the final value of $i$ in Algorithm \ref{algo: Bernoulli factory rand} has an exponential tail, which can be used for establishing that Algorithm \ref{algo: Bernoulli factory non-rand} is fast as defined by Nacu and Peres.

\begin{theorem}
\label{teo: algo non-rand E[N], fast}
Algorithm \ref{algo: Bernoulli factory non-rand} requires an average total number of inputs
\begin{equation}
\label{eq: E N, non-rand}
\E[N] = \frac{f(p)}{p} \left(1 + \frac{2}{p(1-p)}\right).
\end{equation}
In addition, the algorithm satisfies \eqref{eq: exp tail}.
\end{theorem}

The value of $\E[N]$ attained by Algorithm \ref{algo: Bernoulli factory non-rand} could be improved in several ways. Firstly, if $d_i \in [0,1]$ happens to be a dyadic number, i.e.\ its binary expansion has an infinite trail of zeros or ones starting at the $m$-th digit, step \ref{step: algo non-rand, loop} is unnecessary for $j \geq m$ (once $j$ reaches $m$, the output $V_i$ is known to be the repeating digit). This can lead to a lower $\E[N]$ for certain functions. Secondly, step \ref{step: algo non-rand, loop}, which is von Neumann's procedure for obtaining a Bernoulli variable with parameter $1/2$, can be replaced by more efficient approaches; see for example \cite{Elias72} and \cite{Peres92}. Lastly, instead of Algorithm \ref{algo: Bernoulli factory non-rand} a modification of \cite[algorithm 3]{Latuszynski11} could be used, replacing the auxiliary uniform random variable required therein by a procedure similar to step \ref{step: algo rand, U n V n} of Algorithm \ref{algo: Bernoulli factory non-rand}.

\section{Extensions of the algorithms}
\label{part: extensions}

The presented algorithms can be modified in several ways to extend the range of functions that can be simulated. Algorithm \ref{algo: Bernoulli factory rand} will be considered in the following, but the discussion also applies to its non-randomized version given by Algorithm \ref{algo: Bernoulli factory non-rand}.

An obvious modification is to change step \ref{step: algo rand, output} of the algorithm so that instead of $Y = X_i$ it outputs the complementary variable $Y=1-X_i$. This simulates the function $g(p) = 1-f(p)$. The average number of inputs and asymptotic optimality of the modified algorithm are unaffected. Equation \eqref{eq: f 1-p} and conditions \eqref{eq: cond f(p)} and \eqref{eq: cond f'(p)} are modified replacing $f(p)$ by $1-g(p)$, whereas \eqref{eq: c k non-neg} and \eqref{eq: sum c k} are maintained.

The same operation can be applied to the input variables in step \ref{step: algo rand, input}. This allows simulation of functions $g(p) = f(1-p)$, where $f$ satisfies the conditions of the original algorithm; and the simulation is asymptotically optimal for $p \rightarrow 1$.

It is also possible to simulate a function obtained from applying certain operations to two constituent functions $f_1$ and $f_2$. Define the functions $f$ (composition of $f_1$, $f_2$), $g$ (product with complement) and $h$ (convex combination) as follows:
\begin{align}
\label{eq: comb of functions, start}
f(p) & = f_2(f_1(p)) \\
g(p) & = 1-(1-f_1(p)) (1-f_2(p)) \\
\label{eq: comb of functions, end}
h(p) & = \alpha f_1(p) + (1-\alpha) f_2(p) \quad  \text{for } \alpha \in (0,1).
\end{align}

\begin{proposition}
\label{prop: extensions}
Consider $f_1$, $f_2$ that can be expressed as in \eqref{eq: f 1-p}--\eqref{eq: sum c k}.
\begin{itemize}
\item
$f$, $g$ and $h$ can also be expressed as in \eqref{eq: f 1-p}--\eqref{eq: sum c k}.
\item
$f$, $g$ and $h$ satisfy \eqref{eq: cond f(p)} if $f_1$ and $f_2$ do.
\item
$f$, $g$ and $h$ satisfy \eqref{eq: cond f'(p)} if $f_1$ and $f_2$ do.
\end{itemize}
\end{proposition}

By Proposition \ref{prop: extensions}, if $f_1$ and $f_2$ satisfy \eqref{eq: f 1-p}--\eqref{eq: sum c k} Algorithm \ref{algo: Bernoulli factory rand} can be used to directly simulate $f$, $g$ or $h$. Alternatively, it is possible to simulate $f_1$ and $f_2$ separately and then combine the results to obtain the desired function. In the three cases this alternative approach is easily seen to require the same average number of inputs as applying Algorithm \ref{algo: Bernoulli factory rand} to $f$, $g$ or $h$. Consider for example the case of function $f$. In the alternative approach, Algorithm \ref{algo: Bernoulli factory rand} is first applied to simulate $f_1$ with input sequence $\secx$. This produces a sequence of independent Bernoulli variables with parameter $f_1(p)$. Then the algorithm is applied again to simulate $f_2$ on this sequence. The first stage requires $f_1(p)/p$ inputs on average to produce each output. The second uses on average $f_2(f_1(p))/f_1(p)$ outputs of the first stage as inputs. The average number of original inputs is the product of those two numbers, which equals $f(p)/p$.

Of course, other combinations of functions may be realizable, even if the resulting function cannot be simulated directly by a single application of the algorithm. For example, if $f_1$, $f_2$ satisfy the conditions for Algorithm \ref{algo: Bernoulli factory rand}, it is immediate to simulate $f(p)=f_1(p)f_2(p)$ by multiplying the outputs for $f_1$ and $f_2$ (the average number of required inputs can be reduced by not computing the second output if the first is $0$). However, it may not be possible to simulate $f$ directly because its coefficients $c_k$ are not necessarily non-negative. Similarly, given $\alpha \in (0,1)$, multiplying the output for $f_1$ by a Bernoulli variable with parameter $\alpha$ simulates $\alpha f_1$. This allows relaxing the restriction \eqref{eq: sum c k} to $\sum_{k=1}^\infty c_k \le 1$.

\section{Conclusions and future work}
\label{part: concl, future}

An algorithm has been presented that can simulate certain functions $f$ using an average number of inputs that, within the class of fast algorithms (in the sense of \cite{Nacu05}), is asymptotically optimal for $p$ vanishingly small. This algorithm generalizes that given in \cite{Wastlund99} for $f(p) = \sqrt{p}$, uses fewer inputs, and admits a non-randomized version.

In future research, it would be interesting to relax the sufficient condition \eqref{eq: cond f'(p)} for asymptotic optimality (Theorem \ref{teo: E[N] asympt opt}), perhaps using a different proof technique. The class of Bernoulli factories to which Theorems \ref{teo: E[N] bound} and \ref{teo: E[N] asympt opt} apply (namely, fast factories) could be extended if the continuity of $\E[N]$ as a function of $p$ (Proposition \ref{prop: E[N] cont p})
could be proved under more general conditions. It would also be useful to extend the algorithm to a more general class of functions, especially in relation to condition \eqref{eq: c k non-neg}. In this regard, \cite[section 3.1]{Latuszynski11} gives a method to simulate functions with alternating series expansions.

\section{Proofs}
\label{part: proofs}

\subsection{Proof of Proposition \ref{prop: c k d k}}

From \eqref{eq: sum c k} and \eqref{eq: d k c k} it stems that
\begin{align}
\label{eq: prop: c k d k 1}
d_k & = \frac{c_k}{\sum_{j=k}^\infty c_j}, \\
\label{eq: prop: c k d k 2}
1-d_k & = \frac{\sum_{j=k+1}^\infty c_j}{\sum_{j=k}^\infty c_j}.
\end{align}
Combining \eqref{eq: prop: c k d k 1} and \eqref{eq: prop: c k d k 2},
\begin{equation}
d_k \prod_{j=1}^{k-1} (1-d_j) = \frac{c_k}{\sum_{j=k}^\infty c_j} \frac{\sum_{j=2}^\infty c_j}{\sum_{j=1}^\infty c_j} \frac{\sum_{j=3}^\infty c_j}{\sum_{j=2}^\infty c_j} \cdots \frac{\sum_{j=k}^\infty c_j}{\sum_{j=k-1}^\infty c_j} = \frac{c_k}{\sum_{j=1}^\infty c_j} = c_k.
\end{equation}
\qed

\subsection{Proof of Theorem \ref{teo: Pr Y}}

The algorithm ends after taking $n$ inputs producing output $Y=0$ if and only if $X_i=0$, $V_i=0$ for $i \leq n -1$; $V_n=1$; and $X_n = 0$. Since all these variables are independent,
\begin{equation}
\Pr[N=n,\, Y=0] = (1-d_1)(1-d_2)\cdots(1-d_{n-1}) d_n (1-p)^{n},
\end{equation}
which by Proposition \ref{prop: c k d k} equals $c_n (1-p)^n$. Therefore
\begin{equation}
\Pr[Y=1] = 1 - \sum_{n=1}^\infty \Pr[N=n,\, Y=0] = 1 - \sum_{n=1}^\infty c_n (1-p)^n = f(p).
\end{equation}
\qed

\subsection{Proof of Theorem \ref{teo: E[N], exp tail}}

The algorithm requires at least $n$ inputs if and only if $X_i=0$, $V_i=0$ for $i \leq n-1$; that is,
\begin{equation}
\label{eq: Pr N geq n}
\Pr[N \geq n] = (1-d_1)(1-d_2)\cdots(1-d_{n-1}) (1-p)^{n-1}.
\end{equation}
Thus
\begin{equation}
\E[N] = \sum_{n=1}^\infty \Pr[N \geq n] = \sum_{n=1}^\infty (1-d_1)(1-d_2)\cdots(1-d_{n-1}) (1-p)^{n-1}.
\end{equation}
Making use of Proposition \ref{prop: c k d k} and equation \eqref{eq: d k c k},
\begin{equation}
\E[N] = \sum_{n=1}^\infty \frac{c_n}{d_n} (1-p)^{n-1} = \sum_{n=1}^\infty (1-p)^{n-1} \left(1-\sum_{j=1}^{n-1} c_j \right) = \frac{1}{p} - \sum_{n=1}^\infty \sum_{j=1}^{n-1} c_j (1-p)^{n-1}.
\end{equation}
Since all the terms in the double series are non-negative, the order of summation can be changed. This gives, taking into account \eqref{eq: f 1-p},
\begin{equation}
\label{eq: proof theo E[N] 1}
\E[N] = \frac{1}{p} - \sum_{j=1}^\infty \, \sum_{n=j+1}^\infty c_j (1-p)^{n-1} = \frac{1}{p} - \sum_{j=1}^\infty \frac{c_j (1-p)^j} p = \frac{f(p)}{p}.
\end{equation}

Using the fact that all coefficients $d_i$ are upper-bounded by $1$, \eqref{eq: Pr N geq n} yields
\begin{equation}
\label{eq: Pr N > n}
\Pr[N > n] = (1-d_1)(1-d_2)\cdots(1-d_n) (1-p)^n \leq (1-p)^n,
\end{equation}
and thus \eqref{eq: exp tail} holds with $A=1$, $\beta = 1-p<1$.
\qed

\subsection{Proof of Proposition \ref{prop: E[N] cont p}}

Let $S$ be an open subset of $(0,1)$, and let $f: S \to (0,1)$ be a function. Consider an arbitrary, randomized Bernoulli factory $\mathcal B$ for $f$ based on the sequence $\secx$ with parameter $p$, randomizing sequence $\secu$, stopping functions $\tau_i(x_1,u_1;\ldots,x_i,u_i)$, and output functions $\gamma_i(x_1,u_1;\ldots,x_i,u_i)$, $i \in \mathbb N$, all assumed to be measurable.

For clarity, the following definitions will be used, which explicitly show the dependence of certain probabilities on $p$: $\Phi_n(p) = \Pr[N \geq n]$; $\varphi_n(p) = \Pr[N=n]$; and $\varphi_{n,y}(p) = \Pr[N=n,\, Y=y]$ for $y \in \{0,1\}$.

The randomized factory $\mathcal B$ can be replaced by an equivalent, non-randomized sequential procedure $\mathcal B'$ that produces the same output using an input sequence $\secz$ defined by $Z_i = X_i + U_i$. The equivalence is clear from the fact that $X_i$ and $U_i$ can be retrieved from $Z_i$ as $X_i = \lfloor Z_i \rfloor$, $U_i = Z_i - \lfloor Z_i \rfloor$. The stopping functions of $\mathcal B'$, denoted by $\tau'_i$, are related to those of $\mathcal B$ as $\tau'_i(z_1,\ldots,z_i) = \tau_i(\lfloor z_1 \rfloor, z_1-\lfloor z_1 \rfloor; \ldots; \lfloor z_i \rfloor, z_i-\lfloor z_i \rfloor)$. Each $\tau'_i$ is measurable because $\tau_i$ is. Similarly, $\gamma'_i(z_1,\ldots,z_i)$ is defined as $\gamma_i(z_1-\lfloor z_1 \rfloor,\ldots,z_i-\lfloor z_i \rfloor)$. Let the random variable $N$ represent the number of $Z_i$ inputs used by $\mathcal B'$ (or of $X_i$ inputs used by $\mathcal B$). The sequence $\secz$ will be said to have parameter $p$ if the underlying $\secx$ sequence has parameter $p$. Note that the stopping time $N$ is randomized from the point of view of $\secx$, but is non-randomized with respect to $\secz$.

For a given $n$, the random variables $Z_1,\ldots,Z_n$ are independent; and their joint probability density, with respect to Lebesgue measure, is $\lambda_n(z_1,\ldots,z_n;p) = \prod_{i=1}^n \lambda(z_i;p)$ with
\begin{equation}
\label{eq: lambda}
\lambda(z;p) = \begin{cases}
1-p & \text{if } z\in(0,1) \\
p & \text{if } z\in(1,2).
\end{cases}
\end{equation}
Defining $r = \lfloor z_1\rfloor + \cdots + \lfloor z_n\rfloor$, $\lambda_n(z_1,\ldots,z_n;p)$ can be expressed as $p^r(1-p)^{n-r}$, and
\begin{equation}
\label{eq: der lambda n}
\frac{\partial \lambda_n(z_1,\ldots,z_n;p)}{\partial p} = (r-np)p^{r-1}(1-p)^{n-r-1} = \frac{r-np}{p(1-p)} \lambda_n(z_1,\ldots,z_n;p).
\end{equation}

Let $R_n \subseteq (0,2)^n$ be the set of all $n$-tuples $(z_1,\ldots,z_n)$ such that $\mathcal B'$ with inputs $Z_1 = z_1$,\ldots, $Z_n=z_n$ stops at $N=n$. This set is defined by the conditions $\tau'_1(z_1) = 0$, $\tau'_2(z_1,z_2) = 0$, \ldots, $\tau'_n(z_1,\ldots,z_n) = 1$. Since the stopping functions $\tau'_i$ are measurable, the region $R_n$ is measurable too, and
\begin{equation}
\label{eq: Pr n lambda}
\varphi_n(p) = \int_{R_n} \lambda_n(z_1,\ldots,z_n;p) \, \mathrm d z_1 \cdots \mathrm d z_n.
\end{equation}
Similarly, given $y \in \{0,1\}$, let $R_{n,y}$ denote the region of all $n$-tuples $(z_1,\ldots,z_n)$ such that $\mathcal B'$ with inputs $Z_1 = z_1,\ldots  Z_n=z_n$ stops at $N=n$ producing the output $Y = y$. Clearly $R_n = R_{n,0} \cup R_{n,1}$, where $R_{n,0}$ and $R_{n,1}$ are disjoint. The function $\gamma'_n$ is measurable because the functions $\gamma_i$ are; and thus ${\gamma'_n}^{-1}(\{y\})$ is a measurable set for $y \in \{0,1\}$. The intersection of this set with $R_n$ is precisely the region $R_{n,y}$, which is thus measurable, and
\begin{equation}
\begin{split}
\label{eq: varphi n 2}
\varphi_{n,1}(p)
&= \int_{R_n} \gamma_n'(z_1,\ldots,z_n) \lambda_n(z_1,\ldots,z_n;p) \, \mathrm d z_1 \cdots \mathrm d z_n \\
&= \int_{R_{n,1}} \lambda_n(z_1,\ldots,z_n;p) \, \mathrm d z_1 \cdots \mathrm d z_n.
\end{split}
\end{equation}

The probability $\varphi_n(p) = \Pr[N = n]$ is computed as the sum of $2^n$ terms $\pi(x_1,\ldots,x_n)$, each associated to an $n$-tuple $(x_1,\ldots,x_n) \in \{0,1\}^n$, where $\pi(x_1,\ldots,x_n)$ is the probability that the first $n$ inputs of $\secx$ are $x_1,\ldots,x_n$ and $\mathcal B$ stops at $N=n$. With $r = x_1 + \cdots + x_n$, this can be expressed in terms of the stopping functions $\tau_i$ as follows:
\begin{equation}
\label{eq: pi}
\begin{split}
& \pi(x_1,\ldots,x_n) = p^r (1-p)^{n-r} \cdot \Pr[\tau_1(x_1,U_1)=0] \\
&\quad \cdot \Pr[\tau_2(x_1,U_1;x_2,U_2)=0 \,|\, \tau_1(x_1,U_1)=0] \cdots \\
&\quad \cdot \Pr[\tau_{n-1}(x_1,U_1;x_2,U_2;\ldots;x_{n-1},U_{n-1})=0 \,|\, \tau_{n-2}(x_1,U_1;x_2,U_2;\ldots;x_{n-2},U_{n-2})=0] \\
&\quad \cdot \Pr[\tau_n(x_1,U_1;x_2,U_2;\ldots;x_n,U_n)=1 \,|\, \tau_{n-1}(x_1,U_1;x_2,U_2;\ldots;x_{n-1},U_{n-1})=0].
\end{split}
\end{equation}
The factors in \eqref{eq: pi} involving stopping functions do not depend on $p$. Thus $\pi(x_1,\ldots,x_n)$ is a polynomial in $p$, and so is $\varphi_n(p)$. This implies that $\Phi_n(p) = 1-\sum_{i=1}^{n-1} \varphi_i(p)$ is also a polynomial, and thus a continuous function of $p$. Therefore, taking into account that
\begin{equation}
\label{eq: E N series 1}
\E[N] = \sum_{n=1}^\infty \Phi_n(p),
\end{equation}
to establish the finiteness of $\E[N]$ and its continuity as a function of $p \in S$ it suffices to show that the above series converges uniformly on any interval $[\zeta,\eta] \subset S$.

Consider $\zeta, \eta>0$ arbitrary such that $[\zeta,\eta] \subset S$. By assumption $\mathcal B$ is fast, that is, $\Phi_{n+1}(p)$ satisfies a bound of the form \eqref{eq: exp tail}. According to \cite[proposition 21]{Nacu05}, this bound can be made uniform on $[\zeta,\eta]$. Thus, there exist $A$ and $\beta$ independent of $p$ such that $\Phi_n(p) \leq A\beta^{n-1}$ for all $p \in [\zeta,\eta]$. This implies that, given $t \in \mathbb N$,
\begin{equation}
\label{eq: E N series 2}
\sum_{n=t}^\infty \Phi_n(p) \leq A \sum_{n=t}^\infty \beta^{n-1} = \frac{A\beta^{t-1}}{1-\beta}
\end{equation}
for all $p \in [\zeta,\eta]$. Since $\beta < 1$, the bound \eqref{eq: E N series 2} can be made as small as desired by choosing $t$ large enough, which shows that the series \eqref{eq: E N series 1} converges uniformly to $\E[N]$ on $[\zeta,\eta]$. This implies \cite[theorem 2.11]{Fleming77} that $\E[N]$ is a continuous function of $p$ on that interval. Since $\zeta$ and $\eta$ are arbitrary, $\E[N]$ is continuous on $S$.
\qed

\subsection{Proof of Theorem \ref{teo: E[N] bound}}

Let $\secz$, $\mathcal B'$, $\lambda_n(z_1,\ldots,z_n;p)$, $\lambda(z;p)$, $R_n$, $R_{n,y}$, $\varphi_n(p)$ and $\varphi_{n,y}(p)$ be defined as in the proof of Proposition \ref{prop: E[N] cont p}.

Applying the sequential procedure $\mathcal B'$ to inputs taken from $\secz$ with parameter $p$ produces a Bernoulli variable $Y$ with parameter $f(p)$. The key idea of the proof is to consider $Y$ as an estimator of $f(p)$. Namely, the variance of $Y$ is 
\begin{equation}
\label{eq: var Y}
\Var[Y] = f(p)(1-f(p)).
\end{equation}
Substituting this into inequality \eqref{eq: C-R} from Lemma \ref{lem: C-R}, to be proved below, will give \eqref{eq: E[N] bound}. This lemma, in turn, uses the result in Lemma \ref{lem: conv unif der}.

\begin{lemma}
\label{lem: conv unif der}
Under the hypotheses of Theorem \ref{teo: E[N] bound}, the series $\sum_{n=1}^\infty \partial \varphi_{n,1}(p)/\partial p$ converges uniformly on any interval $[\zeta,\eta] \subseteq S$.
\end{lemma}

\begin{proof}
Consider an arbitrary interval $[\zeta,\eta] \subseteq S$. It will be shown first that differentiation under the integral sign is possible in \eqref{eq: varphi n 2} for $p \in [\zeta,\eta]$. This requires checking certain regularity conditions, so that Leibniz's rule can be applied. To this end, consider an open interval $(\zeta',\eta')$ such that $[\zeta,\eta] \subset (\zeta',\eta') \subset S$, with $\zeta'>0$, $\eta'<1$. To ensure the validity of Leibniz's rule it suffices to show that the subintegral function and its derivative are continuous and bounded for $(z_1,\ldots,z_n) \in R_{n,1}$, $p \in (\zeta',\eta')$ \cite[page~237]{Fleming77}.

From \eqref{eq: der lambda n} it is seen that $\partial \lambda_n(z_1,\ldots,z_n;p) /\partial p$ has a discontinuity when any $z_i$ approaches $1$ (which causes a jump in $r$), but is continuous for $(z_1,\ldots,z_n)$ within each of the $2^n$ hypercubes $(x_1,x_1+1) \times \cdots \times (x_n,x_n+1)$, $x_i \in \{0,1\}$, $i = 1,\ldots,n$. These hypercubes will be denoted as $H_{x_1,\ldots,x_n}$. In order to differentiate under the integral sign in \eqref{eq: varphi n 2}, the region $R_{n,1}$ needs to be divided into $2^n$ sets resulting from the intersection of $R_{n,1}$ with one of the hypercubes $H_{x_1,\ldots,x_n}$. The resulting sets are disjoint and measurable with respect to Lebesgue measure. The integral in \eqref{eq: varphi n 2} can thus be expressed as
\begin{equation}
\label{eq: varphi n partes}
\varphi_{n,1}(p) = \sum_{x_1,\ldots,x_n} \int_{R_{n,1} \cap H_{x_1,\ldots,x_n}} \lambda_n(z_1,\ldots,z_n;p) \, \mathrm d z_1 \cdots \mathrm d z_n.
\end{equation}
Within each of the $2^n$ integration regions in \eqref{eq: varphi n partes} and for $p \in (\zeta',\eta')$, the function $\partial \lambda_n(z_1,\ldots,z_n;p) / \partial p$ is bounded and continuous, because the region is contained in a single $H_{x_1,\dots,x_n}$ and $p$ is in $(\zeta',\eta')$ (thus $r$ does not have any jumps and $p$ is bounded away from $0$ or $1$). The function $\lambda_n(z_1,\ldots,z_n;p)$ is bounded and continuous too. Therefore Leibniz's rule can be applied to each integral in \eqref{eq: varphi n partes}; that is,
\begin{equation}
\label{eq: der varphi n parts}
\frac{\partial \varphi_{n,1}(p)} {\partial p} = \sum_{x_1,\ldots,x_n} \int_{R_{n,1} \cap H_{x_1,\ldots,x_n}}  \frac{\partial \lambda_n(z_1,\ldots,z_n;p)}{\partial p} \, \mathrm d z_1 \cdots \mathrm d z_n
\end{equation}
and $\partial \varphi_{n,1}(p)/\partial p$ is a continuous function of $p \in [\zeta,\eta]$.

The uniform convergence of $\sum_{n=1}^\infty \partial \varphi_{n,1}(p)/\partial p$ on $[\zeta, \eta]$ is easily obtained using \eqref{eq: der lambda n} and \eqref{eq: der varphi n parts} as follows. The term $r-np$ in \eqref{eq: der lambda n} can be bounded as $|r-np| < n$. In addition, for $p \in [\zeta, \eta]$ the term $p(1-p)$ is lower-bounded by $\zeta(1-\eta)$. Therefore
\begin{equation}
\label{eq: der lambda n bound}
\left| \frac{\partial \lambda_n(z_1,\ldots,z_n;p)}{\partial p} \right| \leq \frac{|r-np|}{p(1-p)} \lambda_n(z_1,\ldots,z_n;p) < \frac{n}{\zeta(1-\eta)} \lambda_n(z_1,\ldots,z_n;p).
\end{equation}
Using \eqref{eq: der varphi n parts} and \eqref{eq: der lambda n bound} and taking into account \eqref{eq: varphi n 2},
\begin{equation}
\label{eq: der phi cota}
\begin{split}
\left| \frac{\partial \varphi_{n,1}(p)} {\partial p} \right| & <
\sum_{x_1,\ldots,x_n} \int_{R_{n,1} \cap H_{x_1,\ldots,x_n}}  \left| \frac{\partial \lambda_n(z_1,\ldots,z_n;p)}{\partial p} \right| \, \mathrm d z_1 \cdots \mathrm d z_n \\
& \leq \frac{n}{\zeta(1-\eta)} \sum_{x_1,\ldots,x_n} \int_{R_{n,1} \cap H_{x_1,\ldots,x_n}} \lambda_n(z_1,\ldots,z_n;p) \, \mathrm d z_1 \cdots \mathrm d z_n \\
& = \frac{n}{\zeta(1-\eta)} \int_{R_{n,1}}  \lambda_n(z_1,\ldots,z_n;p) \, \mathrm d z_1 \cdots \mathrm d z_n = \frac{n \varphi_{n,1}(p)}{\zeta(1-\eta)} \leq \frac{n \varphi_n(p)}{\zeta(1-\eta)}.
\end{split}
\end{equation}

From Proposition \ref{prop: E[N] cont p}, $\E[N] = \sum_{n=1}^\infty n \varphi_n(p)$ is a continuous function of $p$ on $[\zeta,\eta]$, and then by Dini's theorem
\cite[page 29]{Hirsch99} the series $\sum_{n=1}^\infty n \varphi_n(p)$ converges uniformly on that interval. By the dominated uniform convergence theorem \cite[theorem 5.9]{Bressoud07}, the bound \eqref{eq: der phi cota} then implies that $\sum_{n=1}^\infty \partial \varphi_{n,1}(p)/\partial p$ converges uniformly on $[\zeta,\eta]$.
\end{proof}

\begin{lemma}
\label{lem: C-R}
Under the hypotheses of Theorem \ref{teo: E[N] bound}, the variable $Y$ produced by the Bernoulli factory satisfies the following (sequential Cram\'er-Rao) bound for $p \in S$:
\begin{equation}
\label{eq: C-R}
\Var[Y] \geq \frac{(f'(p))^2 p(1-p)}{\E[N]},
\end{equation}
where $\E[N]$ is the average number of inputs used for producing $Y$.
\end{lemma}

\begin{proof}
It suffices to prove that, given an arbitrary interval $(\zeta,\eta) \subset S$ with $\zeta>0$, $\eta<1$, \eqref{eq: C-R} holds for all $p \in (\zeta,\eta)$. This will be done using Wolfowitz's extension of the Cram\'er-Rao bound to sequential estimators \cite{Wolfowitz47}, which particularized to $\mathcal B'$ and $Y$ will give the desired result.

Consider $(\zeta,\eta) \subset S$, $\zeta>0$, $\eta<1$. The sequential version of the Cram\'er-Rao bound will hold on this interval if the five regularity conditions enunciated in \cite[section 3]{Wolfowitz47} are satisfied. The first condition specifies that $p$ must belong to an open interval, which is indeed the case.

The second regularity condition requires that $\partial \lambda(z;p)/\partial p$ exist for all $p$ and almost all $z$, and that $\E[\partial \log \lambda(Z;p)/\partial p] = 0$ and $\E[(\partial \log \lambda(Z;p)/\partial p)^2] > 0$ for all $p \in (\zeta,\eta)$, where $Z$ is a generic variable from the sequence $\secz$. This easily follows from \eqref{eq: lambda} by computing
\begin{equation}
\frac{\partial \lambda(z;p)}{\partial p} = \begin{cases}
-1 & \text{if } z\in(0,1) \\
1 & \text{if } z\in(1,2).
\end{cases}
\end{equation}
and
\begin{equation}
\label{eq: der log lambda}
\frac{\partial \log \lambda(z;p)}{\partial p} = \begin{cases}
-1/(1-p) & \text{if } z\in(0,1) \\
1/p & \text{if } z\in(1,2).
\end{cases}
\end{equation}
Taking into account that $\Pr[Z \in (0,1)] = 1-p$ and $\Pr[Z \in (1,2)] = p$, from \eqref{eq: der log lambda} it stems that $\E[ \partial \log \lambda(z;p) / \partial p] = 0$ and
\begin{equation}
\label{eq: E der log lambda squared}
\E\left[ \left( \frac{\partial \log \lambda(z;p)} {\partial p} \right)^2 \right] = \frac 1 {1-p} + \frac 1 {p} = \frac 1 {p(1-p)},
\end{equation}
which is strictly positive as required.

The third condition requires that, for $n \in \mathbb N$ and for variables $Z_1, \ldots, Z_n$ belonging to sequence $\secz$,
\begin{equation}
\E\left[ \left( \sum_{j=1}^n \left| \frac{\partial \log \lambda(Z_j;p)} {\partial p} \right| \right)^2 \right]
\end{equation}
exists for all $p \in (\eta,\zeta)$. This is satisfied because, according to \eqref{eq: der log lambda}, $|\partial \log \lambda(Z_j;p)/\partial p|$ is upper-bounded by $\max\{1/\zeta, 1/(1-\eta)\}$ for $p \in (\zeta,\eta)$.

The fourth condition states that $|\gamma'_n(z_1,\ldots,z_n) \partial \lambda_n(z_1,\ldots,z_n;p)/\partial p|$ be bounded by a measurable function of $z_1,\ldots,z_n$ with finite integral on $R_n$. This clearly holds because $|\gamma'_n(z_1\ldots,z_n)| \leq 1$ and, for $p \in (\zeta,\eta)$, \eqref{eq: der lambda n} implies
\begin{equation}
\left| \frac{\partial \lambda_n(z_1,\ldots,z_n;p)} {\partial p} \right| < \frac{n}{p(1-p)} < \frac{n}{\zeta(1-\eta)}.
\end{equation}

Lastly, the fifth regularity condition postulates the uniform convergence of the series $\sum_{n=1}^\infty \partial \varphi_{n,1}(p)/\partial p$ on the interval $(\zeta,\eta)$. This is established by Lemma \ref{lem: conv unif der}.

Since the regularity conditions are satisfied, Wolfowitz's single-parameter inequality \cite[equation (4.5)]{Wolfowitz47} holds for $p \in (\zeta,\eta)$, namely
\begin{equation}
\label{eq: from Wolfowitz}
\Var[Y] \geq \frac{\left( {\mathrm d \E[Y]} / {\mathrm d p} \right)^2} {\E[N] \E\left[ \left( {\partial \log \lambda(z;p)} / {\partial p} \right)^2 \right]}.
\end{equation}
The estimator $Y$ is unbiased, that is, $\E[Y] = f(p)$. Thus
\begin{equation}
\label{eq: der E Y}
\frac{\mathrm d \E[Y]}{\mathrm d p} = f'(p).
\end{equation}
Substituting \eqref{eq: E der log lambda squared} and \eqref{eq: der E Y} into \eqref{eq: from Wolfowitz} yields \eqref{eq: C-R}. Since the interval $(\zeta,\eta) \subseteq S$ is arbitrary, this holds for any $p \in S$.
\end{proof}

As indicated, the result in Theorem \ref{teo: E[N] bound} readily follows from \eqref{eq: var Y} and Lemma \ref{lem: C-R}.
\qed

\subsection{Proof of Proposition \ref{prop: ex}}

\paragraph{1} $f(p) = p^a$, $a \in (0,1)$ can be written as in \eqref{eq: f 1-p} with $c_k$ given by \eqref{eq: coef p a}. These coefficients clearly satisfy \eqref{eq: c k non-neg}, and \eqref{eq: sum c k} also holds because $\lim_{p \rightarrow 0}f(p) = 0 = 1- \sum_{k=1}^\infty c_k$. Condition \eqref{eq: cond f a b}, which implies \eqref{eq: cond f(p)} and \eqref{eq: cond f'(p)},  is obviously satisfied with the same $a$ as in the function definition and $b=1$.

\paragraph{2} For $f$ given by \eqref{eq: ex frac sqrt},
\begin{equation}
\label{eq: prop ex 2 1}
1-f(1-p) = 1 - \frac{2 \sqrt{1-p}}{1+\sqrt{1-p}} = \frac{1-\sqrt{1-p}}{1+\sqrt{1-p}} = 2 \frac{1-\sqrt{1-p}} p - 1.
\end{equation}
Consider the series expansion of $\sqrt p$ and let $c'_k$ denote its coefficients: $1-\sqrt{1-p} = \sum_{k=1}^\infty c'_k p^k$, where, according to \eqref{eq: c_k sqrt}, $c'_1 = 1/2$. Substituting into \eqref{eq: prop ex 2 1}, 
\begin{equation}
1-f(1-p) = 2 \sum_{k=1}^\infty c'_k p^{k-1} - 1 =  2\sum_{k=1}^\infty c'_{k+1} p^k.
\end{equation}
Thus $f(p)$ can be expressed as in \eqref{eq: f 1-p} with coefficients $c_k = 2c'_{k+1}$, which are positive and sum to $1$. On the other hand, condition \eqref{eq: cond f a b} is satisfied with $a=1/2$, $b=2$.

\paragraph{3} From \cite[equation 4.6.32]{Abramowitz70},
\begin{equation}
\begin{split}
\arcsech z &= \log \left(\frac 1 z + \sqrt{\frac 1 {z^2} -1} \right) = \log\frac 2 z - \sum_{k=1}^\infty \frac{(2k-1)!!}{(2k)!!}\frac{z^{2k}}{2k} \\
&= \log\frac 2 z - \sum_{k=1}^\infty \binom{2k}{k} \frac{z^{2k}}{2^{2k+1} k} \quad \text{for } |z|<1.
\end{split}
\end{equation}
Substituting $z^2 = 1-p$ and rearranging,
\begin{equation}
\log(1+\sqrt p) = \log 2 - \sum_{k=1}^\infty \binom{2k}{k} \frac{(1-p)^k}{2^{2k+1} k} \quad \text{for } p \in (0,1).
\end{equation}
Therefore
\begin{equation}
\label{eq: prop ex 3 1}
\log_2(1+\sqrt p) = 1 - \sum_{k=1}^\infty \binom{2k}{k} \frac{(1-p)^k}{2^{2k+1} k \log 2} \quad \text{for } p \in (0,1).
\end{equation}
Comparing with \eqref{eq: f 1-p}, the coefficients $c_k$ for this function are seen to be positive. Since $\lim_{p \rightarrow 0} \log_2(1+\sqrt p) = 0$, the coefficients sum to $1$. Lastly, $\lim_{p \rightarrow 0} \log_2(1+\sqrt p) / \sqrt p$ is easily seen to be $\log_2 e$, and thus \eqref{eq: cond f a b} holds.

\paragraph{4} From \cite[pages 41 and 42]{Grosswald78} (see also \cite{OEIS-A144301}),
\begin{equation}
\label{eq: prop ex 4 1}
e^{(1-\sqrt{1-2zt})/z} = \sum_{k=0}^\infty \frac{y_{k-1}(z)}{k!} t^k
\end{equation}
where $y_j(z)$ are the Bessel polynomials, which satisfy \cite[pages 18 and 20]{Grosswald78}
\begin{equation}
\label{eq: prop ex 4 2}
y_{-1}(z) = y_0(z) = 1, \qquad y_{j}(z) = (2j-1)zy_{j-1}(z) + y_{j-2}(z).
\end{equation}
Setting $z=1$ and substituting $t = p/2$ in \eqref{eq: prop ex 4 1}, and then using \eqref{eq: prop ex 4 2} gives
\begin{equation}
e^{1-\sqrt{1-p}} = 1 + \sum_{k=1}^\infty \frac{y_{k-1}(1)}{2^k\,k!} p^k.
\end{equation}
This implies that
\begin{equation}
\label{eq: prop ex 4 3}
f(p) = \frac{1-e^{-\sqrt{p}}}{1-e^{-1}} = 1 - \sum_{k=1}^\infty \frac{y_{k-1}(1)}{(e-1)2^k\,k!} (1-p)^k.
\end{equation}
Comparing with \eqref{eq: f 1-p} and taking into account \eqref{eq: prop ex 4 2}, the coefficients $c_k$ are seen to be positive. Their sum is $1$ because $\lim_{p \rightarrow 0} f(p) = 0$; and \eqref{eq: cond f a b} holds because, as is easily seen, $\lim_{p \rightarrow 0} f(p)/\sqrt p = e/(e-1)$.

\paragraph{5} Consider the series expansion \cite[equation 4.1.26]{Abramowitz70}
\begin{equation}
\label{eq: log p series}
\log p = -\sum_{k=1}^\infty \frac{(1-p)^k}{k} \quad \text{ for } p \in (0,1).
\end{equation}
Multiplying by $1-p$ in \eqref{eq: log p series}, subtracting \eqref{eq: log p series}, collecting same-power terms and rearranging gives
\begin{equation}
\label{eq: p - p log p series}
f(p) = p-p\log p = 1 - \sum_{k=2}^\infty \frac{1}{k(k-1)} (1-p)^k \quad \text{ for } p \in (0,1).
\end{equation}
Identifying terms in \eqref{eq: f 1-p} and \eqref{eq: p - p log p series}, it is clear that \eqref{eq: c k non-neg} is satisfied. Since $\lim_{p \rightarrow 0} f(p) = 0$, \eqref{eq: sum c k} holds as well.

Clearly $\lim_{p \rightarrow 0} f(p)/p = \lim_{p \rightarrow 0} (1-\log p) = \infty$, which establishes \eqref{eq: cond f(p)}. On the other hand, the derivative of $f(p)$ is $f'(p) =- \log p$. Thus
\begin{equation}
\lim_{p \rightarrow 0} \frac{f'(p) p}{f(p)} = \lim_{p \rightarrow 0} \frac{\log p}{\log p -1} = 1,
\end{equation}
and therefore \eqref{eq: cond f'(p)} holds.
\qed

\subsection{Proof of Theorem \ref{teo: E[N] asympt opt}}

The proof uses a standard argument based on the definition of $\Omega$-type asymptotic growth and on Theorem \ref{teo: E[N] bound}.

By assumption \eqref{eq: cond f'(p)}, there exist $C, \delta >0$ such that
\begin{equation}
\label{eq: teo E[N] asympt opt proof 1}
f'(p) \geq C {f(p)} / {p} \quad \text{ for } p \in S \cap (0,\delta).
\end{equation}
The function $f$ satisfies the hypotheses of Theorem \ref{teo: E[N] bound}. Substituting \eqref{eq: teo E[N] asympt opt proof 1} into \eqref{eq: E[N] bound},
\begin{equation}
\label{eq: teo E[N] asympt opt proof 2}
\E[N] \geq \frac{C^2 f(p)(1-p)}{p(1-f(p))}  \quad \text{ for } p \in S \cap (0,\delta).
\end{equation}

The claimed result is equivalent to the statement that there are $C', \delta'>0$ such that $\E[N] > C' f(p)/p$ for all $p \in S \cap (0,\delta')$. Taking $C' = C^2/2$, $\delta' = \min\{\delta, 1/2\}$, from \eqref{eq: teo E[N] asympt opt proof 2} it stems that
\begin{equation}
\label{eq: teo E[N] asympt opt proof 3}
\E[N] \geq \frac{2C' f(p)(1-p)}{p(1-f(p))} > \frac{C' f(p)}{p} \quad \text{ for } p \in S \cap (0,\delta'),
\end{equation}
which establishes the result.
\qed

\subsection{Proof of Theorem \ref{teo: algo non-rand E[N], fast}}

The following result about geometric random variables will be used.

\begin{lemma}
\label{lem: geom exp tail}
Geometric random variables have exponential tails.
\end{lemma}

\begin{proof}
Given $\theta \in (0,1)$, consider a random variable $W$ with $\Pr[W=w] = \theta(1-\theta)^{w-1}$, $w \in \mathbb N$. Computing $\Pr[W \geq w+1] = (1-\theta)^w$ shows that \eqref{eq: exp tail} is satisfied with $A=1$, $\beta = 1-\theta<1$.
\end{proof}

Steps $2$--$4$
of Algorithm \ref{algo: Bernoulli factory non-rand} (which are the same as in Algorithm \ref{algo: Bernoulli factory rand} except step $3$) form an (outer) loop on $i$ which is repeated until the exit condition in step $4$ is met. For each iteration of this loop, Algorithm \ref{algo: Bernoulli factory rand} uses one input from $\secx$; whereas Algorithm \ref{algo: Bernoulli factory non-rand} uses that input plus additional ones that are needed for generating $V_i$, as specified by steps \ref{step: algo non-rand, start}--\ref{step: algo non-rand, branch}. 
For a given $i$, the number of iterations of the (inner) loop on $j$ formed by steps \ref{step: algo non-rand, loop} and \ref{step: algo non-rand, branch} is a geometric random variable $K_i$ with parameter $1/2$, and thus $\E[K_i] = 2$. For each $j$, the number of blocks of $2$ inputs required within step \ref{step: algo non-rand, loop} (which can be regarded as a third-level, innermost loop) is a geometric random variable $L_{i,j}$ with parameter $2p(1-p)$, and thus $\E[L_{i,j}] = 1/(2p(1-p))$. Let
\begin{equation}
L_i = L_{i,1}+\cdots+L_{i,K_i}.
\end{equation}
The variables $L_{i,j}$ are i.i.d.\ and independent from $K_i$, and therefore \cite[page 194]{Papoulis91} $\E[L_i] = \E[K_i]\E[L_{i,1}] = 1/(p(1-p))$. Thus, for each $i$ Algorithm \ref{algo: Bernoulli factory non-rand} uses one input from $\secx$ in step $2$ (like Algorithm \ref{algo: Bernoulli factory rand} does), plus $1/(p(1-p))$ $2$-input blocks on average in steps \ref{step: algo non-rand, start}--\ref{step: algo non-rand, branch}. Consequently, Algorithm \ref{algo: Bernoulli factory non-rand} uses on average $1+2/(p(1-p))$ as many inputs as Algorithm \ref{algo: Bernoulli factory rand} does. This proves \eqref{eq: E N, non-rand}.

By Lemma \ref{lem: geom exp tail}, the variables $L_{i,j}$ as well as $K_i$ have exponential tails; and then \cite[proposition 12]{Nacu05} guarantees that $L_i$ has an exponential tail. For each $i$, the iteration formed by steps $2$--$4$ of Algorithm \ref{algo: Bernoulli factory non-rand} requires $1 + 2 L_i$ inputs. The total number of iterations of the outer loop on $i$ coincides with the number of inputs of Algorithm \ref{algo: Bernoulli factory rand}, which has an exponential tail as established by Theorem \ref{teo: E[N], exp tail}. Since $1 + 2 L_i$ also has an exponential tail, applying \cite[proposition 12]{Nacu05} again shows that the total number of inputs used by Algorithm \ref{algo: Bernoulli factory non-rand} has an exponential tail, that is, satisfies \eqref{eq: exp tail}.
\qed

\subsection{Proof of Proposition \ref{prop: extensions}}

Consider functions $f_1(p) = 1-\sum_{i=1}^\infty c_{1,i} (1-p)^i$ and $f_2(p) = 1-\sum_{j=1}^\infty c_{2,j} (1-p)^j$ that satisfy \eqref{eq: f 1-p}--\eqref{eq: sum c k},
and let $f$, $g$ and $h$ be defined as in \eqref{eq: comb of functions, start}--\eqref{eq: comb of functions, end}.

\paragraph{1. Function $f(p)=f_2(f_1(p))$} Identifying coefficients in
\begin{equation}
\label{eq: f comp}
f(p) = f_2(f_1(p)) = 1-\sum_{j=1}^\infty c_{2,j} \left(\sum_{i=1}^\infty c_{1,i} (1-p)^i \right)^j
= 1-\sum_{k=1}^\infty c_k (1-p)^k
\end{equation}
it is seen that $c_k \geq 0$ for $c_{1,i}, c_{2,j} \geq 0$. Also, $\lim_{p \rightarrow 0} f(p) = 0$ because $\lim_{p \rightarrow 0} f_1(p) = \lim_{p \rightarrow 0} f_2(p) = 0$; and thus $\sum_{k=1}^\infty c_k = 1$.

Assume that $f_1$ and $f_2$ satisfy \eqref{eq: cond f(p)}. Conditions \eqref{eq: f 1-p}--\eqref{eq: sum c k}
imply that $f_1$ has an inverse function, such that $p = f_1^{-1}(s)$, and that $p \rightarrow 0$ if and only if $s \rightarrow 0$. Therefore, condition \eqref{eq: cond f(p)} for $f_1$ can be written as
\begin{equation}
\label{eq: lim p s 1}
\lim_{p \rightarrow 0} \frac{f_1(p)}{p} = \lim_{s \rightarrow 0} \frac{s}{f_1^{-1}(s)} = \infty.
\end{equation}
Thus
\begin{equation}
\label{eq: lim p s 2}
\lim_{p \rightarrow 0} \frac{f(p)}{p} = \lim_{p \rightarrow 0} \frac{f_2(f_1(p))}{p} = \lim_{s \rightarrow 0} \frac{f_2(s)}{f_1^{-1}(s)} = \lim_{s \rightarrow 0} \frac{f_2(s)}{s} \frac{s}{f_1^{-1}(s)}.
\end{equation}
Since $f_2$ also satisfies \eqref{eq: cond f(p)}, it follows that $\lim_{p \rightarrow 0} f(p)/p = \infty$, which establishes condition \eqref{eq: cond f(p)} for $f$.

Assuming that \eqref{eq: cond f'(p)} holds for $f_1$ and $f_2$, to prove that it holds for $f$ it suffices to find $C, \delta > 0$ such that $f'(p) > C f(p)/p$ for all $p < \delta$. The derivative of $f$ is
\begin{equation}
f'(p) = f'_2(f_1(p)) f'_1(p).
\end{equation}
Since $f_1$ and $f_2$ satisfy \eqref{eq: cond f'(p)}, there exist $C_1, C_2, \delta_1, \delta_2 > 0$ such that
\begin{equation}
f'(p) \geq C_2 \frac{f_2(f_1(p))}{f_1(p)} C_1 \frac{f_1(p)}{p} = C_1 C_2 \frac{f_2(f_1(p))}{p}.
\end{equation}
for all $p<\min\{\delta_1, f_1^{-1}(\delta_2)\}$. Taking $C = C_1 C_2$ and $\delta = \min\{\delta_1, f_1^{-1}(\delta_2)\}$ establishes the desired result.

\paragraph{2. Function $g(p) = 1-(1-f_1(p)) (1-f_2(p))$} Writing
\begin{equation}
\label{eq: f prod}
g(p) = 1-(1-f_1(p)) (1-f_2(p)) = 1-\sum_{i=1}^\infty \sum_{j=1}^\infty c_{1,i} c_{2,j} (1-p)^{i+j}
= 1- \sum_{k=1}^\infty c_k (1-p)^k
\end{equation}
and identifying coefficients makes it clear that $c_k \geq 0$. On the other hand, it is easily seen that $\lim_{p \rightarrow 0} f(p) = 0$, which implies that $\sum_{k=1}^\infty c_k = 1$.

The function $g$ can be expressed as
\begin{equation}
\label{eq: f1+f2-f1f2}
g(p) = f_1(p)+f_2(p)-f_1(p)f_2(p) = f_1(p)+f_2(p)(1-f_1(p)).
\end{equation}
Taking into account that $\lim_{p \rightarrow 0} f_1(p) = 0$, it follows from \eqref{eq: f1+f2-f1f2} that \eqref{eq: cond f(p)} holds for $g$ if it holds for $f_1$ and $f_2$.

From \eqref{eq: f1+f2-f1f2}, the derivative of $g$ is computed as
\begin{equation}
g'(p)
= f'_1(p) (1-f_2(p)) + f'_2(p) (1-f_1(p)).
\end{equation}
Assuming that $f_1$ and $f_2$ satisfy \eqref{eq: cond f'(p)} and taking into account that they are monotone increasing, it stems that there exist $C_1, C_2 > 0$,  $\delta \in (0,1/2)$ such that for all $p<\delta$
\begin{equation}
\label{eq: der prod ineq}
\begin{split}
g'(p) & \geq C_1 \frac{f_1(p)}{p} (1-f_2(p)) + C_2 \frac{f_2(p)}{p} (1-f_1(p)) \\
& > C_1 (1-f_2(1/2)) \frac{f_1(p)}{p} + C_2 (1-f_1(1/2)) \frac{f_2(p)}{p}.
\end{split}
\end{equation}
Defining $C = \min\{C_1(1-f_2(1/2)), C_2(1-f_1(1/2))\}$ and making use of \eqref{eq: f1+f2-f1f2} again, it stems from \eqref{eq: der prod ineq} that
\begin{equation}
g'(p) > C \frac{f_1(p)+f_2(p)}{p} > C\frac{g(p)}{p}
\end{equation}
for all $p < \delta$, which establishes \eqref{eq: cond f'(p)} for $g$.

\paragraph{3. Function $h(p) = \alpha f_1(p) + (1-\alpha) f_2(p)$} The proof is straightforward.
\qed

\section*{Acknowledgments}

The author wishes to thank professor Krzysztof \L{}atuszy{\'n}ski for his comments on an earlier version of the paper; Mats Granvik for his help with the Taylor expansion of \eqref{eq: ex exp sqrt}; and the anonymous reviewers for very helpful comments and suggestions.


\bibliographystyle{elsarticle-num}

\end{document}